\documentclass[letterpaper, 10pt, conference]{ieeeconf}      % Use this line for a4
\IEEEoverridecommandlockouts                              % This command is only
\usepackage{amsfonts,amsmath,amssymb,bm}
\usepackage{color,cases,float}
\usepackage{graphicx}
\usepackage{hyperref}
\usepackage{psfrag}
\usepackage{eucal}
\usepackage{algorithm}
\usepackage{overpic}
\usepackage{subfigure}
\usepackage{url}
\usepackage[all,cmtip]{xy}
\usepackage{algorithm}
\usepackage[noend]{algpseudocode}

\definecolor{darkblue}{rgb}{0.0,0.0,0.6}
\hypersetup{colorlinks,breaklinks,
	linkcolor=darkblue,urlcolor=darkblue,
	anchorcolor=darkblue,citecolor=darkblue}
\newtheorem{assumption}{Assumption}
\newtheorem{definition}{Definition}
\newtheorem{lem}{Lemma}
\newtheorem{corrolary}{Corrolary}

\newtheorem{theorem}{Theorem}
\newtheorem{remark}{Remark}

\makeatletter
\newcommand{\setassumptiontag}[1]{% \settheoremtag{<tag>}
  \let\oldtheassumption\theassumption% Store \thetheorem
  \renewcommand{\theassumption}{#1}% Redefine it to a fixed value
  \g@addto@macro\endassumption{% At \end{theorem}, ...
    \addtocounter{assumption}{-1}% ...restore theorem counter value and...
    \global\let\theassumption\oldtheassumption}% ...restore \thetheorem
  }
\makeatother

\newcommand{\E}{\mathrm{E}}

\newcommand{\R}{\mathrm{Re}}
\newcommand{\Proj}{\Pi}

\def\blab{\boldsymbol{\lambda}}

\def\R{\mathbb{R}}
\def\Wb{\mathbf{W}}

\def\zb{\mathbf{z}}
\def\bx{\boldsymbol{x}}

\def\ab{\boldsymbol{a}}

\def\mb{\boldsymbol{m}}

\def\s{\sigma}
\def\g{\gamma}

\def\bmu{\boldsymbol{\mu}}

\def\hbmu{\hat{\boldsymbol{\mu}}}

\def\kb{\boldsymbol{k}}
\def\lb{\boldsymbol{l}}

\def\gb{\boldsymbol{g}}

\def\R{\mathbb{R}}
\def\Wb{\boldsymbol{W}}

\def\zb{\mathbf{z}}
\def\xb{\boldsymbol{x}}

\def\bmu{\boldsymbol{\mu}}

\def\Mb{\boldsymbol{M}}
\def\Qb{\boldsymbol{Q}}
\def\bla{\boldsymbol{\lambda}}

\def\Rb{\boldsymbol{R}}
\def\Sb{\boldsymbol{S}}
\def\Ab{\boldsymbol{A}}

\def\r{~}
\def\BibTeX{{\rm B\kern-.05em{\sc i\kern-.025em b}\kern-.08em
		T\kern-.1667em\lower.7ex\hbox{E}\kern-.125emX}}

	\title{Convergence Rate of Payoff-based Generalized Nash Equilibrium Learning }
	\author{Tatiana Tatarenko, Maryam Kamgarpour  \IEEEmembership{Member, IEEE}
    \thanks{T. Tatarenko is with TU Darmstadt, Control Methods and Intelligent Systems Lab (e-mail: tatiana.tatarenko@tu-darmstadt.de).}
		\thanks{M. Kamgarpour  is with EPFL Sycamore Lab (e-mail: maryam.kamgarpour@epfl.ch). }}

\begin{document}

\maketitle
\thispagestyle{empty}
\pagestyle{empty}

%%%%%%%%%%%%%%%%%%%%%%%%%%%%%%%%%%%%%%%%%%%%%%%%%%%%%%%%%%
\begin{abstract}
We consider generalized Nash equilibrium (GNE) problems in games with strongly monotone pseudo-gradients and jointly linear coupling constraints.  We establish the convergence rate of a payoff-based approach intended to learn a variational GNE (v-GNE) in such games. 
While convergent algorithms have recently been proposed in this setting given full or partial information of the gradients, rate of convergence in the payoff-based information setting has been an open problem. Leveraging properties of a game extended from the original one by a dual player, we establish a convergence rate of $O(\frac{1}{t^{4/7}})$ to a v-GNE of the game.

\end{abstract}

%%%%%%%%%%%%%%%%%%%%%%%%%%%%%%%%%%%%%%%%%%%%%%%%%%%%%%%%%%
\section{Introduction}\label{sec:intro}

% motivating GNE
A generalized Nash equilibrium (GNE) problem models settings where multiple decision makers have coupled objective and constraint functions. The applicability of GNE problems spans from economics to engineering systems \cite{facchinei2007generalized}.  Several works have focused on characterizing and computing a solution to the GNE problem given knowledge of the cost functions and constraints, see \cite{jordan2023first} for a recent overview of algorithms.  Our focus here is on computing a solution of a class of GNE problems without prior knowledge on players' cost or constraint functions. We assume only access to the so-called ``payoff"/zeroth-order  information. Namely, players can evaluate their cost and constraint functions based on a  chosen joint action. Our goal is to characterize the rate at which one can learn  a GNE solution based on these function evaluations.

% motivating partial bandit. 
Learning equilibria given payoff-based information  is a well-established field in game theory \cite{pradelski2012learning}. The motivation is that in many practical multi-agent settings, the agents do not know functional form of the objectives and constraints and can only access the values of their respective functions at a played action. Such situations arise, for example, in electricity markets (unknown price functions or constraints), network routing (unknown traffic demands/constraints), and sensor coverage problems (unknown density function on the mission space). Even in the cases where cost and constraint functions are known, a payoff-based approach can be computationally advantageous due to complexity of computing gradients. This latter motivation is in line with works on derivative-free or zeroth-order optimization.

Most past work on computing solution of a GNE problem assumes availability of the gradient information. The general approach to solve this problem class in convex games is to cast it as an equivalent quasi-variational inequality (QVI) problem. In the case in which all agents share the same coupling constraint, a subset of GNE solutions, referred to as variational GNE (v-GNE),  can be found through a variational inequality (VI) problem. Moreover, a v-GNE is shown to be a refinement of the GNE and to possess economic merit~\cite{KULKARNI2012}.  The focus of early algorithm for GNE problems has been on proving \textit{asymptotic} convergence rather than characterizing \textit{convergence rates}~\cite{facchinei2007generalized}. Recently, \cite{jordan2023first,meng2023linear} proposed algorithms for GNE problems with monotone pseudo-gradients and  linear coupling constraints, and established the convergence rates of their respective algorithms. The work of \cite{meng2023linear}  considered the case of all players sharing the same constraint function and derived a primal dual approach with a linear convergence rate, given a strongly monotone pseudo-gradient and full row rank of the constraint's matrix. On the other hand, \cite{jordan2023first} considered the more general class of GNEs where agents can have individually different coupling constraints, derived a convergent algorithm and established its convergence rate.
%of $\tilde{O}(\epsilon^{-1})$ under strongly monotone, and $\tilde{O}(\epsilon^{-1/2})$ under merely monotone cost functions. 
However, the notion of approximate solution in the work above does not imply convergence of the iterates to a GNE solution. Furthermore,  both of the above works required knowledge of the constraints and the costs' gradients.

% zeroth-order methods.
As for payoff-based approaches, for games \textit{without} coupling constraints, past work had established convergence rates  under strong monotonicity of the game pseudo-gradient \cite{drusvyatskiy2022improved}  or strong variational stability of equilibria \cite{TatKamECC24}. While GNE problems with jointly convex constraints can be cast as a VI problem, unfortunately, the resulting VI will not be strongly monotone and the results from the papers above do not apply.
%  and the equilibrium cannot be strongly variationally stable. 
For GNE problems with jointly convex coupling constraints,  
\cite{tat_kam_TAC}  derived a payoff-based convergent algorithm. However, that work assumed that the game admits a potential function. The potential function assumption  simplifies the analysis as it enables one to cast the problem as a zeroth-order constrained optimization problem. The work of \cite{suli} extended this setting to include non-potential games with jointly linear constraints, and presented a Tikhonov regularization approach, similar to those proposed in \cite{tatarenko2019learning} for non-strongly monotone games. However, \cite{suli} did not establish a convergence rate for the algorithm. To the best of our knowledge, deriving convergence rate of payoff-based learning in a GNE problem had not been addressed.

We analyze convergence rate of the payoff-based  algorithm for GNE problems under the standard assumptions of strongly monotone cost functions with Lipschitz continuous gradients and linear coupling constraints. Given jointly linear constraints, the VI formulation of our problem exhibits monotonicity. For general monotone variational inequalities, establishing a convergence rate of iterates is not possible without very specific additional  assumptions~\cite[Chapter 5.8 ]{bakushinsky2012ill}.  To obtain the convergence rate,  we derive a property of the pseudo-gradient in the extended game (with a dual player) which allows us to: 
%1) conclude monotonicity of the extended pseudo-gradient and, thus, applicability of the Tikhonov regularization technique; 
1) estimate the convergence rate of the primal part solution to the regularized problem (distance between $\ab^*_t$ and $\ab^*$, Lemma~\ref{lem:epsApprox}); 2) balance the regularization term with the stochastic sampling and step size parameters so as to bound the iteration distance to the v-GNE (Theorem~\ref{th:main}), which results in the convergence rate of the order $O(1/t^{4/7})$ (Corollary~\ref{cor:1}). 
To the best of our knowledge, our work is the first to establish the convergence rate  of the payoff-based algorithm in a GNE problem formulated for strongly-monotone games with affine shared constraints.

\subsection{Notations. }
The set $\{1,\ldots,N\}$ is denoted by $[N]$. We consider real normed space $\R^D$ and denote by $\R_{+}^D$ its non-negative halfspace. The column vector $\xb\in\R^D$ is denoted by $\xb = [x^1,\ldots, x^D]$. We use superscripts to denoted coordinates of vectors and the player-related functions and sets. 
For any function $f:K\to\R$, $K\subseteq\R^d$, $\nabla_{x^i} f(\xb) = \frac{\partial f(\xb)}{\partial x^i}$ is the partial derivative taken in respect to the $x^i$th variable (coordinate) in the vector argument $\xb\in\R^D$.
 We use $\langle \cdot,\cdot\rangle$ to denote the inner product in $\R^D$.
We use $\|\cdot\|$ to denote the Euclidean norm induced by the standard dot product in $\R^D$. 
A mapping $g:\R^D\to \R^D$ is said to be \emph{strongly monotone} on $Q\subseteq \R^D$ with the constant $\nu$, if for any $u, v\in Q$, $\langle g(u)-g(v), u - v\rangle \ge \nu\|u - v\|^2$; \emph{strictly monotone}, if the strict inequality holds for $\nu=0$ and $u\neq v$, and \emph{merely monotone} if the inequality holds for $\nu=0$.
We use $\Proj_{\Omega}{v}$ to denote the projection of $v\in \R^D$ onto a set $\Omega\subseteq \R^D$.
The  expectation of a random value $\xi$ is denoted by $\E\{\xi\}$. Its conditional expectation with respect to some $\sigma$-algebra $\EuScript F$ is denoted by $\E\{\xi|\EuScript F\}$.
We use the big-$O$ notation, that is, the function $f(x): \R\to\R$ is $O(g(x))$ as $x\to a$ for some $a\in\R$, i.e. $f(x)$ = $O(g(x))$ as $x\to a$, if $\lim_{x\to a}\frac{|f(x)|}{|g(x)|}\le K$ for some positive constant $K$.

\section{Game Settings}
We consider a non-cooperative game $\Gamma$ between $N$ players. Let player $i$, $i \in [N]$, choose its local action from $\R^{d_i}$ and $J^i: \R^D\to\R$ with $D=\sum_{i=1}^Nd_i$ denote her cost function. We consider a \emph{coupling constraint set} $C$ defined by the following system of linear inequalities: 
\begin{align}\label{eq:coupled_cs}
C = \{\ab\in\R^D: \gb(\ab)=K\ab-\lb\le \boldsymbol{0}\},
\end{align}
where $K\in\R^{n\times D}$ is a matrix with columns   $\kb^{i,j} \in \R^n$, for $j\in [d_i]$, $i\in[N]$, $\lb\in\R^{n}$, $\ab = [\ab^1,\ldots,\ab^N]\in\R^D$, with $\ab^i = [a^{i,1},\ldots,a^{i,d_i}]\in\R^{d_i}$ being the action of the player $i$ and $\ab^{-i}$ denoting the joint action of all players except for the player $i$. %Thus, $C = \{\ab\in\R^D: \gb(\ab) = K\ab-\lb\le \boldsymbol{0}\}$ with $K=\sum_{i=1}^N \boldsymbol{e}_i\otimes K^i$ and $\lb = [\lb^1,\ldots,\lb^N]\in\R^{n}$, $n=\sum_{i=1}^Nn_i$.
Thus, the joint action set in the game $\Gamma$ is $C$ and the coupled action set of the agent $i$ by ${C}^i(\ab^{-i}) = \{\ab^i\in\R^{d_i}: K\ab-\lb\le \boldsymbol{0}\}$. 

A \emph{generalized Nash equilibrium} (GNE) in a game $\Gamma$ with coupled actions represents a joint action from which no player has any incentive to unilaterally deviate.
\begin{definition}\label{def:GNE}
 A joint action $\ab^*\in  C$ in the game $\Gamma=\Gamma(N, \{J^i\}, C)$ with the coupling constraints $C$ is called a \emph{generalized Nash equilibrium} (GNE) if $J^i(\ab^{*,i},\ab^{*,-i})\le J^i(\ab^i,\ab^{*,-i})$ for any $i\in[N]$ and $\ab^i\in C^i(\ab^{*,-i})$.
 
Observe that if $C=\R^{D}$ then $ C^i(\ab^{-i})=\{\ab^i: \ab^i\in \R^{d_i}\}$ and any $\ab^*$ for which the above inequality holds is a \emph{Nash equilibrium (NE)}.
 \end{definition}

%We are interested in designing a \emph{payoff-based algorithm}, which converges to a generalized Nash equilibrium in $\Gamma$. The payoff-based information structure implies that  agent $i$, $i \in [N]$, does not know the functional form of $J_i$ or $\gb$ but for a played action $\ab$, it can observe their values. 
We consider convex games as defined below. 
\begin{assumption}\label{assum:convex}
 The game under consideration is \emph{convex}. Namely, for all $i\in[N]$ the cost function $J^i(\ab^i, \ab^{-i})$ is defined on $\R^{D}$, continuously differentiable in $\ab$ and convex in $\ab^i$ for  fixed $\ab^{-i}$. %The coupling constraint function $\gb:\R^{Nd}\to\R^n$ is continuously differentiable and has convex coordinates $g_i(\ab), i\in[n]$.
\end{assumption}

Given differentiable cost functions, we define the game's pseudo-gradient mapping.
\begin{definition}\label{def:mapping}
The mapping $\Mb = [\Mb^1, \ldots, \Mb^N]:\R^{D}\to\R^{D}$, referred to as  the \emph{pseudo-gradient} of  $\Gamma(N,  \{J^i\}, C)$ is defined by
 \begin{align}\label{eq:gamemapping}
 &\Mb^i(\ab) = [M^{i,1}(\ab), \ldots, M^{1,d_i}(\ab)]^{\top}, \cr
 &M^{i,j}(\ab)= \frac{\partial J^i(\ab)}{\partial a^{i,j}}, \; \ab\in  C, j\in[d_i], i\in[N].
 \end{align}
\end{definition}

As we deal with the generalized Nash equilibrium problem in games with shared constraints, the set of \emph{variational generalized Nash equilibria} can be defined as follows \cite{facchinei2007generalized}. 
\begin{definition}\label{def:VGNE}
    The set of \emph{variational generalized Nash equilibria  (v-GNE)} in the game $\Gamma$ is defined as the set of solutions to the following variational inequality problem:
    \begin{align}\label{eq:GVI}
    \mbox{Find }\ab^*\in  C: \, \langle\Mb(\ab^*),\ab-\ab^*\rangle\ge 0, \,\mbox{ for all }\ab\in\EuScript C .\tag{$VI_{(\Mb, C)}$}
\end{align}
\end{definition}
The next lemma states the relation between GNE and v-GNE (see \cite{facchinei2007generalized} for the proof). 
\begin{lem}\label{lem:GNEandVGNE}
    Under Assumption\r\ref{assum:convex}, the set of generalized Nash equilibria in the game $\Gamma$ contains the set of variational generalized Nash equilibria, i.e. v-GNE $\subseteq$ GNE\footnote{Without loss of generality we assume that for any set $\EuScript S$ the following holds: $\emptyset\subseteq \EuScript S$, where $\emptyset$ denotes an empty set.}.   
\end{lem}

In this work, we focus on  games with strongly monotone pseudo-gradients. 
\begin{assumption}\label{assum:strmon}
The pseudo-gradient $\Mb$ is \emph{strongly monotone} over $\R^D$, i.e. there exists a constant $\nu>0$ such that $\langle\Mb(\ab_1)-\Mb(\ab_2),\ab_1-\ab_2\rangle\ge\nu\|\ab_1-\ab_2\|^2$ for any $\ab_1,\ab_2\in\R^D$.
%The coupling constraint function $\gb:\R^{Nd}\to\R^n$ is continuously differentiable and has convex coordinates $g_i(\ab), i\in[n]$.
\end{assumption}
Moreover, we consider the following basic constraint qualification condition.  
  \begin{assumption}\label{assum:Slaters}
The set $ C$ satisfies the \emph{Slater's constraint qualification}. %(see Definition\r\ref{def:slater}).
\end{assumption}

\subsection{Formulation as an unconstrained game}
To be able to deal with the coupling constraints in the game $\Gamma$, we follow the approach of \cite{tat_kam_TAC} and define an augmented game $\Gamma^{\blab}$, for which we introduce a virtual player, a so-called \emph{dual} player indexed by $N+1$, whose action corresponds to the dual variable $\blab\in\R^{n}_{+}$. Furthermore, we extend the cost functions of the initial \emph{(primal)} players to consider them Lagrangian functions. Thus, given a shared constraints defined by a convex set $C = \{\ab: \gb(\ab)\le \boldsymbol{0}\}$, the augmented game $\Gamma^{\blab}$ contains $N+1$ players, and the cost function $U^{i}$ of each primal player $i\in[N]$ is defined as follows: 
\begin{align}\label{eq:J^{pr}r}
U^{i}(\ab,\blab) = J^i(\ab) + \langle \blab, \gb(\ab)\rangle.
\end{align}
The cost function of the dual player $N+1$ is
\begin{align}\label{eq:J_du}
U^{N+1}(\ab,\blab) = -  \langle \blab, \gb(\ab)\rangle.
\end{align}
The joint action set in $\Gamma^{\blab}$ is denoted by $\Ab^{\blab} = \R^D\times\R^n_{+}$, where $\R^n_{+}$ is the action set of the dual player. Hence, $\Gamma^{\blab} = \Gamma^{\blab}(N+1, \{U^{i}\}, \Ab^{\blab})$ is a game with uncoupled constraints. We denote an action of this game by $[\ab,\blab]$ and refer to $\ab\in\R^D$ and $\blab\in\R^{n}_{+}$ as to its \emph{primal and dual part} respectively.  

Next, we formulate the result linking solutions in the initial game $\Gamma$ to the solutions in  game $\Gamma^{\blab}$.
\begin{lem}\label{lem:exist_uncoupled}
 Let $\Gamma(N, \{J^i\}, C)$ be a game with shared constraints defined by a convex set $C = \{\ab: \, \gb(\ab)\le 0\}$. % for which Assumptions\r\ref{assum:convex},~\ref{assum:strmon}, and \ref{assum:Slaters} hold. 
 Then,
 \begin{itemize}
 \item[1)] given Assumptions~\ref{assum:convex} and~\ref{assum:Slaters},  if $[\ab^*,\bla^*]$ is an NE in $\Gamma^{\blab}$, then $\ab^*$ is a v-GNE of $\Gamma$. Moreover, if $\ab^*$ is a v-GNE of $\Gamma$, then there exists $\blab^*\in\R^n_{+}$ such that $[\ab^*,\bla^*]$ is an NE in $\Gamma^{\blab}$,
 \item[2)] given Assumptions~\ref{assum:convex},~\ref{assum:strmon} and~\ref{assum:Slaters}, there exists an NE in $\Gamma^{\blab}$,
 \item[3)] given Assumption~\ref{assum:strmon}, if $[\ab_1^*,\bla_1^*]$ and $[\ab_2^*,\blab_2^*]$ are two NEs in $\Gamma^{\blab}$, then $\ab_1^* = \ab_2^*$,
 \item[4)] for any NE $[\ab^*, \bla^*]$ in $\Gamma^{\blab}$ there exists a constant $\Lambda > 0$ such that $\|\bla^*\|\le \Lambda$.
 \end{itemize}
 \end{lem}
The proof is provided in Appendix~\ref{app:exist_uncoupled}.
The above lemma implies that one can aim to find a solution (NE) to the game $\Gamma^{\blab}$ with uncoupled action sets in the search for a GNE in the initial game $\Gamma$, that is,  by solving $\Gamma^{\blab}$, one achieves a v-GNE. Moreover, under Assumption~\ref{assum:strmon}, as stated above, this v-GNE exists and is unique. 

Let us now take into account the  specific setting of  affine shared constraints, i.e. $\gb(\ab) = K\ab-\lb$.
In this case the pseudo-gradient $\Wb = [\Wb^{1},\ldots,\Wb^{N},\Wb^{N+1}]:\R^{D+n}\to\R^{D+n}$ of the game $\Gamma^{\blab}$ is defined as follows.
 \begin{align}\label{def:mapping_ext}
  &\Wb^{i}(\ab, \blab)=[W^{i,1}(\ab, \blab), \ldots, W^{i,d_i}(\ab, \blab)], \cr
 &\qquad\mbox{where }W^{i,j}(\ab,\blab)= M^{i,j}(\ab) + \frac{\partial \langle\bla, K\ab-\lb\rangle}{\partial a^{i,j}}\cr
 &= M^{i,j}(\ab) + \langle \blab,\kb^{i,j}\rangle ,\; j\in[d_i], \quad i\in[N],\cr
 &\Wb^{N+1}(\ab, \blab)= -K\ab+\lb.
 \end{align}
\begin{remark}
The above pseudo-gradient and thus, $T^\lambda$, is monotone but not strongly monotone. 
\end{remark}
In the following analysis we will need to analyze the  part of the mapping $\Wb$ corresponding to the primal players, namely: 
\begin{align}\label{eq:primalM}
\Wb^{pr}(\ab,\blab) = [\Wb^{1}(\ab,\blab),\ldots,\Wb^{N}(\ab,\blab)]\in\R^D.
\end{align}

We notice that, according to~\eqref{def:mapping_ext}, the pseudo-gradient $\Wb$ of the augmented game $\Gamma^{\blab}$ is linear in the dual action $\blab$. Thus, $\Gamma^{\blab}$ is a convex game and its Nash equilibria  are characterized by solutions to the following variational inequality \cite{FaccPang1}:  Find $[\ab^*,\blab^*]\in\Ab^{\blab}$ such that
\begin{align}\label{eq:VI1}\tag{$VI$}
    \langle\Wb(\ab^*,\blab^*),[\ab,\blab]-[\ab^*,\blab^*]\rangle\ge 0, \, \forall [\ab,\blab]\in\Ab^{\blab} .
\end{align} 
However, the pseudo-gradient $\Wb$ of the game $\Gamma^{\blab}$ in~\eqref{def:mapping_ext} is not strongly monotone in the game's joint action set $\Ab^{\blab}$. On the other hand, due to the \emph{linear structure of the coupling constraints}, it is straightforward to verify that the following relation holds for any two actions $ [\ab_1,\blab_1]\in\Ab^{\blab}$, $[\ab_2,\blab_2]\in\Ab^{\blab}$ of the augmented game $\Gamma^{\blab}$: 
\begin{align}\label{eq:aug-psedo}
    \langle\Wb(\ab_1,\blab_1)-\Wb(\ab_1,\blab_1),[\ab_1,\blab_1]&-[\ab_2,\blab_2]\rangle \cr&\ge \nu\|\ab_1-\ab_2\|^2.
\end{align}
The relation above is essential for providing the analysis of the payoff-based procedure presented in the next sections. 

\subsection{Regularized extended game properties}
In the next section, we will present a payoff-based procedure learning a Nash equilibrium in the game $\Gamma$ which coincides with the primal part of a Nash equilibrium in the augmented game $\Gamma^{\blab}$ with uncoupled actions. Thus, we will develop a procedure solving the variational inequality~\eqref{eq:VI1}. However, we can not use a procedure imitating standard gradient descent one, as the pseudo-gradient $\Wb$ is not strongly monotone, but merely monotone (see~\eqref{eq:aug-psedo}). For this reason, we will regularize the corresponding mapping. 

We will use the idea of Tikhonov regularization  and add the time-varying term $\varepsilon_t I_{0,n}$ to the mapping $\Wb$, where $I_{0,n}\in\R^{D+n\times D+n}$ is a diagonal matrix with the first $D$ zero-elements and the last $n$ positive elements on the diagonal. We let the positive elements be equal to $1$. Thus, the following regularized pseudo-gradient is constructed: 
$$    \Wb_t(\ab,\blab) = \Wb(\ab,\blab) + [\underbrace{0,\ldots, 0}_{D}, \varepsilon_t\blab].$$
Due to strong monotonicity of $\Mb(.)$ from Assumption~\ref{assum:strmon}, for any $ [\ab_1,\blab_1]\in\Ab^{\blab}$, $[\ab_2,\blab_2]\in\Ab^{\blab}$ the following relation holds:
$\langle\Wb_t(\ab_1,\blab_1)-\Wb_t(\ab_2,\blab_2),[\ab_1,\blab_1]-[\ab_2,\blab_2]\rangle \ge \nu\|\ab_1-\ab_2\|^2 + \varepsilon_t\|\blab_1-\blab_2\|$.
Thus, the regularized mapping is strongly monotone and one can use the idea of the standard gradient descent step to solve the following regularized variational inequality:   Find $[\ab_t^*,\blab_t^*]\in\Ab^{\blab}$ such that
\begin{align}\label{eq:VI2} 
 \langle\Wb_t(\ab_t^*,\blab_t^*),[\ab,\blab]-[\ab_t^*,\blab_t^*]\rangle\ge0, \,\forall [\ab,\blab]\in\Ab^{\blab} \tag{$VI_t$}.
\end{align}
Before characterizing the solutions to the variational inequalities above, we make  the following assumption on the cost functions' behavior at infinity.
\begin{assumption}	\label{assum:infty}
	Each function $J^{i}(\bx) = O(\|\bx\|^{2})$ as $\|\bx\|\to\infty$.
\end{assumption}
\begin{remark}\label{rem:Lip}
    The assumption above is required, as we will consider a payoff-based approach to v-GNE learning, and to this end, will estimate gradients based on randomized sampling, basesd on a 
    a distribution with an unbounded support. Moreover, Assumption~\ref{assum:infty} will play an important role in proving boundedness of the iterates in the learning procedure.
    Observe that  Assumption~\ref{assum:infty} is equivalent to the assumption that the functions $\Mb^i$, $i\in[N]$, are \emph{Lipschitz continuous on $\R^{D}$} with some constant $L>0$. The latter assumption is standard in the literature on distributed generalized Nash equilibria seeking \cite{jordan2023first}.
\end{remark}
%\begin{assumption}\label{assum:CG_grad} 
%The vector-functions $\Mb^i$, $i\in[N]$, are \emph{Lipschitz continuous on %$\R^{D}$} with some constant $L>0$.
% \end{assumption}
 
The following lemma characterizes the relation between solutions of ~\eqref{eq:VI1} and the solution of ~\eqref{eq:VI2}. 
\begin{lem}\label{lem:epsApprox}
    Let the mapping $\Wb(\ab,\blab)$ be as in~\eqref{def:mapping_ext} with $\Mb(\ab)$ satisfying Assumptions~\ref{assum:strmon} and~\ref{assum:infty}. 
    Let $[\ab^*,\blab^*]$ be a minimal norm solution to~\eqref{eq:VI1} and $[\ab^*_t,\blab^*_t]$ be the solution to~\eqref{eq:VI2}. Then, as $\varepsilon_t\to 0$, $[\ab^*_t,\blab^*_t]$ converges to $[\ab^*,\blab^*]$ and, moreover,
    $$ \|\ab^*-\ab^*_t\|\le \varepsilon_t\frac{\|\blab^*\|L}{\|K\|\nu}.$$
\end{lem}
Our proof is provided in Appendix~\ref{app1}.

The lemma above estimates the  rate with which the primal part of the  solution to the regularized~\eqref{eq:VI2} converges to the primal part of the solution to the variational inequality~\eqref{eq:VI1}. We emphasize here that such characterization is obtained due to  relation~\eqref{eq:aug-psedo} which holds, if the pseudo-gradient of the game $\Gamma$ is strongly monotone. 

The next lemma connects solutions of two consecutive ~\eqref{eq:VI2} and will be  leveraged for our convergence rate derivation. 
\begin{lem}\label{lem:t_vs_t-1}
Let the mapping $\Wb(\ab,\blab)$ be as in~\eqref{def:mapping_ext} with $\Mb(\ab)$ satisfying Assumption~\ref{assum:strmon}. 
    Let $[\ab^*_t,\blab^*_t]$ and $[\ab^*_{t-1},\blab^*_{t-1}]$ be solutions to~\eqref{eq:VI2} and $(VI_{t-1})$ respectively. Then the following relation holds: 
    $
        \|\ab^*_t-\ab^*_{t-1}\|^2 = O\left(\frac{(\varepsilon_t-\varepsilon_{t-1})^2}{\varepsilon_{t}}\right)$, $        \|\blab^*_t-\blab^*_{t-1}\|^2 = O\left(\frac{(\varepsilon_t-\varepsilon_{t-1})^2}{\varepsilon^2_{t}}\right)$.
\end{lem}
See Appendix~\ref{app1} for the proof.

Next,  we will develop a \emph{payoff-based} procedure to learn a Nash equilibrium in the \emph{convex (not strongly) monotone} augmented game $\Gamma^{\blab}$ and characterize its convergence rate.

\section{Payoff-based Learning algorithm}\label{sec:procedure}

The \emph{payoff-based} settings assume the following information in the system: each player $i$ in the initial game $\Gamma$ has access to the value of her cost function $J^i(t) := J^i(\ab(t))$ at every time step $t$, evaluated at given the corresponding joint action $\ab(t)$. Moreover, the value of the constraint function, that is, the value $\gb(t):=K\ab(t)-\lb$, is revealed to each player $i\in[N]$ at every $t$. Note that neither the joint action $\ab(t)$ nor the closed form of the cost and shared constraint functions are available to the agents.

\subsubsection{Algorithm iterates}
For each time $t$,  let us denote the actions of the primal player $i$ and the dual player $N+1$ in the game $\Gamma^{\bla}$,  by $\ab^i(t)$ and $\blab(t)$ respectively. Let $\mb^i(t)$ be some estimate of the pseudo-gradient's elements of the primal players in the game $\Gamma^{\bla}$ at time $t$. In words, $\mb^i(t)$ estimates the vector $\Wb^i(\ab(t),\blab(t))$, $i\in[N]$ (see~\eqref{def:mapping_ext}).% (see the definition~\eqref{eq:J^{pr}r} of the cost functions of the primal players in $\Gamma^{\bla}$) 

In the payoff-based approach, denote  the iterates corresponding to the primal player $i$, $i\in[N]$, and the dual player $N+1$ be denoted by  $\bmu^i$  and  $\blab$. These iterates are updated  as follows:
\begin{align}
	\label{eq:alg}
	&\bmu^i(t+1)=\bmu^i(t)-\gamma_t\mb^i(t),\\ \nonumber
        &\blab(t+1)=\Proj_{\R^{n}_{+}}[\blab(t)-\gamma_t(-\gb(t)+\varepsilon_t\blab(t))],
   \end{align}
where $\bmu^i(0)\in \R^{d_i}$, $\blab(0)\in\R^{n}_{+}$ are  arbitrary finite vectors, $\g_t$ is the step size, and the choice of $\ab(t)$ will be described below).
%Moreover, as defined before, the following holds for the dual player: $\Ab^{N+1} = \R^n_{+}$. 

Note that according to the discussion in the previous section, the procedure in~\eqref{eq:alg} mimics the gradient descent play in  game $\Gamma^{\blab}$ and intends to achieve a Nash equilibrium in it. The regularization term $\varepsilon_t\blab(t)$ is added only in the update of the dual players. The step size $\gamma_t$ should be chosen based on the regularization parameter of the game, as well as the bias and variance of the pseudo-gradient estimate $\mb^i(t)$. The term $\mb^i(t)$ is obtained using payoff-based feedback as described below.

\subsubsection{Action choice and gradient estimations}
We estimate the unknown gradients using randomized sampling technique. In particular, we use the Gaussian distribution for sampling inspired by \cite{Thatha,NesterovSpokoiny}. 
Given $\bmu^i(t)$, let the primal player $i$ sample her action as the random vector $\ab^i(t)$ according to the multivariate normal distribution $\EuScript N(\bmu(t) = [\mu^{i,1}(t),\ldots,\mu^{i,d_i}(t)]^{\top},\sigma_t)$ with the  density function:
\begin{align}\label{eq:density}
	p^i&(\bx^i;\bmu^i(t),\sigma_{t})= \frac{1}{(\sqrt{2\pi}\sigma_{t})^{d_i}}e^{-\sum_{k=1}^{d_i}\frac{(x^{i,k}-\mu^{i,k}(t))^2}{2\sigma^2_{t}}}.
\end{align}
%Thus, the joint action is $\ab(t) =[\ab^1(t),\ldots,\ab^N(t)]^{\top}\in \R^{D}$. 
According to the information setting, the value of the cost function in the game $\Gamma^{\blab}$ at $\ab(t)$ and $\bla(t)$ (see definition~\eqref{eq:J^{pr}r}), denoted by $U^{i}(t): =U^{i}(\ab(t),\blab(t))=J^i(\ab(t)) + \langle \blab(t), \gb(\ab(t))\rangle$,
is revealed to each primal player $i\in[N]$. 

{\begin{remark}
    Note that the value $U^{i}(t)$ is available to the agent $i$, since, according to the information settings,  agent $i$ has access to the values $J^i(t) = J^i(\ab(t))$ and $\gb(t)=\gb(\ab(t))$. On the other hand,   the dual vector $\blab(t)$ is either broadcast by a central coordinator (virtual dual player) or is locally updated by each primal player $i\in[N]$, given $\blab^i(0) = \blab(0)$ for any player $i$ (indeed, given the same initial dual vector $\blab(0)$ for any player $i$, the iteration \eqref{eq:alg} can be run locally by each primal player resulting in a common dual vector  $\blab(t)$ for each such player at each time $t$).
\end{remark} }

In this work, we intend to  guarantee a bounded variance of gradient estimations and, thus, following  the discussion in~\cite{duchi2015optimal}, we focus on the so called two-point setting for gradient estimations at each iteration $t$. This means that each player $i$ makes two queries: a query corresponding to the chosen action $\ab(t)$ and another query of the cost function $J^i$ and the constraint function $g^i$ at the point
$\bmu(t)$.
Hence, there is an extra piece of information available to each player, namely the value of the cost function evaluated at the mean of the distribution: 
$
U^{i}_0(t): = U^{i}(\bmu(t),\blab(t)) = J^i(\bmu(t)) + \langle \blab(t), \gb(\bmu(t))\rangle
$.
Then, each player uses the following estimation of the local pseud-gradient $\Mb^i(\cdot)$ at the point  $\bmu(t)$:
	\begin{align}\label{eq:est_Gd2}
		\mb^i(t) = (U^{i}(t) - U^{i}_0(t))\frac{{\ab^i(t)} -\bmu^i(t)}{\sigma^2_t}.
	\end{align}

\subsubsection{Properties of the  gradient estimators}
We provide insight into the procedure defined by Algorithm~\eqref{eq:alg}  by deriving an analogy to a  stochastic gradient descent algorithm. Denote 
\begin{align}\label{eq:densityfull}
p( \bx; \bmu, \sigma)=\prod_{i=1}^{N}p^i(\bx^i;\bmu^i,\sigma)
\end{align}
as the  density function of the joint distribution of players' query points $\ab$, given  $\bmu =[\bmu^1,\ldots,\bmu^N]$ (see~\eqref{eq:density} for the definition of the individual density function $p^i$). For any $\sigma > 0$, $i\in[N]$, and given any $\blab\in\R^{n}_{+}$ define $ \tilde{U}^{i}_{\sigma}(\bmu,\blab) : \R^{D} \rightarrow \R$ as
$
\tilde{U}^{i}_{\sigma} (\bmu,\blab)= \int_{\mathbb R^{D}}U^{i}(\bx,\blab)p( \bx; \bmu, \sigma)d\bx$.
Thus, $\tilde{U}^{i}_{\sigma}$, $i\in[N]$, is the $i$th primal player's cost function in the mixed strategies of the game $\Gamma^{\blab}$, where the strategies are sampled from the Gaussian distribution with the density function in~\eqref{eq:densityfull}.
%We can now show that the second term inside the projection in \eqref{eq:regpl} is a sample of the gradient of this cost function $\tilde{J}_i$ with respect to the mixed strategies.
For $i\in[N]$ define $\tilde{\Wb}^{i}_{\sigma} (\cdot)=[\tilde W^{i,1}_{\sigma}(\bmu), \ldots, \tilde W^{i,d_i}_{\sigma}(\bmu)]^{\top}$
as the $d_i$-dimensional mapping with the following elements:
$
\tilde W^{i,k}_{\sigma} (\bmu)=\frac{\partial {\tilde U^{i}_{\sigma}(\bmu,\blab)}}{\partial \mu^{i,k}}$, $k\in[d_i]$.

Next, let $\Rb^{i}(t) = \mb^{i}(t) - \tilde{\Wb}^{i}_{\sigma_t} (\bmu(t),\blab(t))$
and $\Qb^i(t)=\tilde\Wb^{i}_{\sigma_t}(\bmu(t))-\Wb^{i}(\bmu(t))$, $i\in[N]$, evaluate the bias between the  pseudo-gradient, its counterpart in the mixed strategies, and its estimation. With the above definitions, the joint update rule of  primal players in~\eqref{eq:alg}  is equivalent to:
\begin{align}
\label{eq:pbavmuQ}
\bmu(t+1) =\bmu(t)& -\gamma_t\big({\Wb}^{pr}(\bmu(t),\blab(t))+\Qb(t)+\Rb(t)\big).
\end{align}
In the above, we use the notation $\Wb^{pr}(\cdot,\cdot)$ defined in~\eqref{eq:primalM} as well as  the notations $\bmu(t) = [\bmu^1(t),\ldots,\bmu^N(t)]$, $\Qb(t) = [\Qb^1(t),\ldots,\Qb^N(t)]$,  and $\Rb(t) = [\Rb^1(t),\ldots,\Rb^N(t)]$.
In the next section, devoted to convergence analysis of Algorithm~\eqref{eq:alg}, we will use  representation~\eqref{eq:pbavmuQ} for the updates of the primal-related iterates comprised by $\bmu(t)$. According to~\eqref{eq:alg}, the dual-related iterates, namely ${\blab}(t)$, is updated as follows, given $\Sb(t) =  K(\bmu(t) - \ab(t))$: 
\begin{align}\label{eq:dual}
	\blab(t+1)=\Proj_{\R^{n}_{+}}[\blab(t)-\gamma_t(-K\bmu(t) &+ \Sb(t) + \lb+\varepsilon_t\blab(t))].
   \end{align}
%We also decompose the vector-state $\bmu(t)$ into two components: $\bmu(t)=[\bmu(t), \bmu(t)]\in\Ab^{\blab}$, where the first component  $\bmu(t) = (\bmu_1(t),\ldots, \bmu_N(t))\in\Ab$ comprises the joint state of the primal players.

Now, we present the main properties of the terms  $\Qb(t)$, $\Sb(t)$,  and $\Rb(t)$.
Let $\EuScript F_{t}$ be the $\sigma$-algebra generated by the random variables $\{\bmu(k),\blab(k),\ab(k)\}_{k\le t}$. The next lemmas characterize the stochastic term $\Rb^i(t)$, $i\in[N]$.
\begin{lem}\label{lem:sample_grad}
\cite[Lemma~1]{TatKamECC24} Given Assumptions\r\ref{assum:convex} and~\ref{assum:infty}, almost surely $\Rb^{i}(t) = {\mb}^{i}(t) - \E\{{\mb}^{i}(t)|\EuScript F_t\}$ and, thus,
$\E\{\Rb^{i}(t)|\EuScript F_t\} = 0$, $i\in[N]$.
\end{lem}
%For the proof, please refer to  . 

\begin{lem}\label{lem:Rsq}
Under Assumptions~\ref{assum:convex} and~\ref{assum:infty},  given $\lim_{t\to\infty}\sigma_t = 0$, 
$\E\{\|\Rb^{i}(t)\|^2 | \EuScript F_t\} =  O(f^i(\bmu(t),\blab(t))$,
  where $f^i(\bmu(t),\blab(t))$ is a function with a quadratic dependence on both $\bmu(t)$ and $\blab(t)$.
 \end{lem}
See Appendix~\ref{app:Rsq} for the proof.

We emphasize that the statement above is due to quadratic dependence of the cost functions at infinity (Assumption~\ref{assum:infty}). We will use the  estimate for the term $\|\Rb^{i}(t)\|^2$ provided by Lemma~\ref{lem:Rsq} in the proof of the iterates' uniform boundedness. 

%Next, we characterize the second moment of the terms $\Sb^i(t)$ and $\Qb^i(t)$.

\begin{lem}\label{lem:Sterm2moment}
\cite[Lemma~2]{tat_kam_TAC} Let Assumptions~\ref{assum:convex} and~\ref{assum:infty} hold. Choose $\sigma_t$ such that $\lim_{t\to\infty}\sigma_t = 0$. Then almost surely
$
    \E\{\|\Sb(t)\|^2 | \EuScript F_t\}  =O\left(\sigma_t^2\right),
$ and $\E\{\|\Qb^{i}(t)\|^2| \EuScript F_t\} = O(\s^2_t)$.
\end{lem}
%The proof can be found in that of . 

\section{Main result}
In this section we provide the main result on convergence of Algorithm~\eqref{eq:alg} and its convergence  rate. 
%We leverage the procedure's representation~\eqref{eq:pbavmuQ}-\eqref{eq:dual} and use the results provided in the previous subsection on its terms $\Qb(t)$,  $\Rb(t)$, and $\Sb(t)$. 
From the technical point of view, the novelty in analysis is in a four-step convergence rate estimation approach.
We use  Lemma~\ref{lem:Rsq} to conclude almost sure uniform boundedness of the algorithm's iterates $\zb(t)=[\bmu(t),\blab(t)]$ at the first step, whereas the next two steps of the proof focus on  estimating the expected distance between the vector $\zb(t)$ and the solution  of ~\eqref{eq:VI2}, namely, $\zb_t^*= [\ab^*_t, \blab^*_t]$. This and Lemma \ref{lem:t_vs_t-1} implies  the convergence rate of $\E\|\bmu(t) - \ab_{t-1}^*\|$. Finally, the last step leverages Lemma~\ref{lem:epsApprox} to conclude the convergence rate, in expectation, of the iterates $\bmu(t)$ to the unique v-GNE $\ab^*$.

%The intuition behind the analysis is as follows. First, we prove uniform boundedness of the iterate vector $\zb(t)=[\bmu(t),\blab(t)]$ over $t$, which also provides us with the estimation on the distance between the vector $\zb(t)$ and a regularized solution $\zb_t^*= [\ab^*_t, \blab^*_t]$. This distance estimation is required for the next step. %Note that this first step is analogous to the proof of Lemma 3 in \cite{tatarenko2019learning}. The next steps, in their turn, are novel. 
%At the second step, we analyze convergence rate of the expected distance $\|\zb(t)-\zb_t^*\|$. Introduction of the regularization term $\varepsilon_t\blab(t)$ allows for not just establishing its convergence to 0 but also upper bounding the rate. The third step consists in using the rate with which $\E\|\blab(t) - \blab_t^*\|$ tends to 0 to specify the rate of convergence for the primal part, namely for the term $\E\|\ab(t) - \ab_t^*\|$. This specification can be achieved due to a structure of the game $\Gamma^{\blab}$ described by the relation~\eqref{eq:aug-psedo}. Finally, the fourth step uses the result provided in Lemma~\ref{lem:epsApprox} to conclude the convergence rate, in expectation, of the iterates $\bmu(t)$ to the unique variational generalized Nash equilibrium $\ab^*$ of the game $\Gamma$.

\begin{theorem}\label{th:main}
Let Assumptions~\ref{assum:convex}-\ref{assum:infty} hold. Then there exists a unique v-GNE in  game $\Gamma$. Let the joint action $\ab^*$ be such v-GNE. 
Choose the parameters in Algorithm~\eqref{eq:alg} as follows:
$\gamma_t = \frac{G}{t^{g}}$, $\varepsilon_t=\frac{E}{t^{e}}$,  
      $\sigma_t=\frac{S}{t^{s}}$,
where $g,e,s>0$ with $s+g>1$, $g+e<1$, and $g>1/2$.
Let  $h = \min\{2-g-e,  g+s, 2g\}$.
Then for the iterates in Algorithm~\eqref{eq:alg} the following holds: 
$
    \E\|\bmu(t)-\ab^*\|^2 =O\left(\frac{1}{t^{\min\{2e,h-g\}}}\right)$.
\end{theorem}
\begin{proof}
Existence and uniqueness of a v-GNE in the game $\Gamma$ follows from Lemma~\ref{lem:exist_uncoupled} assertions 2)-3).
To obtain the convergence rate, we focus on the representation~\eqref{eq:pbavmuQ} and~\eqref{eq:dual} of Algorithm~\eqref{eq:alg}. %We also make use of the following notation: $\zb(t) = [\bmu(t),\blab(t)]$.

\textbf{Step 1. Boundedness of iterates.} Our approach is to use the result from  Theorem 2.5.2 in \cite{NH}, as stated in Theorem \ref{app_bound}Appendix~\ref{th:bound} of this paper. In particular, we aim to show conditions of Theorem~\ref{th:bound}) hold for  $X(t) = \zb(t)$, $\xb=\zb$, $V(t,\xb) = \|\xb-\zb_{t-1}^*\|^2$   to conclude the  almost sure boundedness of the norm $\|\zb(t)-\zb_{t-1}^*\|$, and thus, the almost sure   boundedness of $\|\zb(t)\|$ for any $t$.

Taking into account stochastic nature of the terms  $\Qb(t)$, $\Sb(t)$, and $\Rb(t)$, their properties  (see Lemmas~\ref{lem:sample_grad}-\ref{lem:Sterm2moment}), we conclude that the procedure~\eqref{eq:pbavmuQ}-\eqref{eq:dual} can be considered as a perturbed gradient-based method to solve the variational inequality~\eqref{eq:VI2}. Thus, in the following we will estimate the distance between the algorithm's iterate at time $t+1$ and the solution $\zb^*_t=[\ab^*_t,\blab^*_t]$ to~\eqref{eq:VI2}, which exists and is unique for any $t$ (due to strong monotonicity of the regularized mapping $\Wb_t$). We obtain
%Before doing so, we recall the following fixed-point property of solutions to variational inequalities. 
%\[\zb^*_t = \Proj_{\Ab^{\blab}}[\zb^*_t - \gamma_t\Mb]\]
\begin{align}\label{eq:ineq1}
&\|\zb(t+1)-\zb_t^*\|^2\le \|\bmu(t) -\gamma_t\big(\Wb^{pr}(\zb(t))+\Qb(t)\cr
&\qquad\qquad\qquad +\Rb(t)) - \ab_t^*\|^2\cr
&+\|\blab(t)-\gamma_t(-K\bmu(t) + \Sb(t) + \lb+\varepsilon_t\blab(t)) - \blab_t^*\|^2\cr
&\le2\gamma_t\langle\Wb_t(\zb_t^*),\zb(t) - \zb_t^*\rangle+\|\zb(t) - \zb_t^*\|^2\cr
&-2\gamma_t\langle\Wb^{pr}(\zb(t)), \bmu(t)-\ab^*_t\rangle \cr
&-2\gamma_t\langle-K\bmu(t) +  \lb + \varepsilon_t\blab(t),\blab(t)-\blab^*_t\rangle\cr 
& -2\gamma_t\langle\Qb(t) +\Rb(t),\bmu(t) - \ab_t^*\rangle\cr
&+\gamma_t^2\|\Wb^{pr}(\zb(t))+\Qb(t) +\Rb(t)\|^2 -2\gamma_t\langle\Sb(t),\blab(t)-\blab^*_t\rangle\cr
&+\gamma_t^2\|-K\bmu(t) + \Sb(t) + \lb+\varepsilon_t\blab(t)\|^2,
\end{align}
where the first inequality is due to the non-expansive property of the projection operator, whereas in the second inequality we used the fact that $\zb_t^*$ solves $VI_t$ and, thus, $\langle\Wb_t(\zb_t^*),\zb(t) - \zb_t^*\rangle\ge 0$. 
We notice that according to the definition of the mapping $\Wb_t$, the following holds for the third and fourth terms of the right hand side in~\eqref{eq:ineq1}:
\begin{align}\label{eq:eq1_1}
    &-2\g_t(\langle\Wb^{pr}(\zb(t)),\bmu(t)-\ab^*_t\rangle\cr
    &\qquad\qquad+ \langle-K\bmu(t) +  \lb + \varepsilon_t\blab(t),\blab(t)-\blab^*_t\rangle) \cr
    &=-2\g_t\langle\Wb_t(\zb(t)),\zb(t)-\zb_t^* \rangle.
\end{align}
Moreover, $\langle\Wb_t(\zb(t))-\Wb_t(\zb_t^*),\zb(t) - \zb_t^*\rangle \ge \nu\|\bmu(t)-\ab_t^*\|^2 + \varepsilon_t\|\blab(t)-\blab_t^*\|$.
Plugging this last inequality together with~\eqref{eq:eq1_1} into~\eqref{eq:ineq1}, we conclude that 
\begin{align}\label{eq:ineq2}
&\|\zb(t+1)-\zb_t^*\|^2\le(1-2\gamma_t\nu)\|\bmu(t)-\ab_t^*\|^2 \cr
&\qquad\qquad\qquad+ (1-2\gamma_t\varepsilon_t)\|\blab(t)-\blab_t^*\|^2\cr
%&\qquad+2\gamma_tL\rho_t\|\bmu(t) - \ab_t^*\| + 2\gamma_t\|K\|\rho_t\|\blab(t)-\blab_t^*\|\cr
&-\gamma_t\langle\Qb(t)+\Rb(t),\bmu(t) - \ab_t^*\rangle\cr
&+\gamma_t^2\|\Wb^{pr}({\zb}(t))+\Qb(t)+\Rb(t)\|^2-2\gamma_t\langle\Sb(t),\blab(t)-\blab^*_t\rangle\cr
&\qquad\qquad\qquad+\gamma_t^2\|-K\bmu(t) + \Sb(t) + \lb+\varepsilon_t\blab(t)\|^2.
\end{align}
Next, we take the conditional expectation with respect to $\EuScript F_t$ of  both sides of the inequality above to obtain: 
    \begin{align}\label{eq:CondExp}
    &\E\{\|\zb(t+1)-\zb_t^*\|^2|\EuScript F_t\}\le(1-2\gamma_t\nu)\|\bmu(t)-\ab_t^*\|^2 \cr
    &+ (1-2\gamma_t\varepsilon_t)\|\blab(t)-\blab_t^*\|^2\cr
    &+\gamma_t\E\{(\|\Qb(t)\|)|\EuScript F_t\}\|\bmu(t) - \ab_t^*\|+2\gamma_t^2\|\Wb^{pr}({\zb}(t))\|^2\cr
    &+2\gamma_t^2\E\{(\|\Qb(t)\|^2+\|\Rb(t)\|^2)|\EuScript F_t\}\cr
    &+2\gamma_t\E\{\|\Sb(t)\||\EuScript F_t\}\|\blab(t)-\blab^*_t\|+2\gamma_t^2(\|K\|^2\|\bmu(t)\|^2\cr
    &\qquad\qquad+\|\lb\|^2+\varepsilon^2_t\|\blab(t)\|^2 + \E\{\|\Sb(t)\|^2|\EuScript F_t\} ),
    \end{align}
where we used the Causchy-Schwarz inequality, and Lemma~\ref{lem:sample_grad} implying that $\E\{\langle\Rb(t),\bmu(t) - \ab_t^*\rangle | \EuScript F_t\} = \boldsymbol{0}$.
   Moreover, Lemma~\ref{lem:Rsq} and  the definition of $\Wb^{pr}$ (see~\eqref{eq:primalM}) together with Assumption~\ref{assum:infty} imply respectively that  
    $
        \E\{\|\Rb_c(t)\|^2|\EuScript F_t\} =  O(\|\zb(t)-\zb^*_{t-1}\|^2)
    $
    and
        $\|\Wb^{pr}({\zb}(t))\|^2 = O(\|\zb(t)-\zb^*_{t-1}\|^2)$.
        Thus, by applying these relations as well as the results in  \eqref{lem:Sterm2moment} to the inequality~\eqref{eq:CondExp}, we obtain: 
    \begin{align}\label{eq:CondExp1}
    &\E\{\|\zb(t+1)-\zb_t^*\|^2|\EuScript F_t\}\le(1-2\gamma_t\nu)\|\bmu(t)-\ab_t^*\|^2 \cr
    &\qquad\qquad\qquad+ (1-2\gamma_t\varepsilon_t)\|\blab(t)-\blab_t^*\|^2\cr
    &+O\left(\gamma_t\sigma_t+\g_t^2\right)(\|\zb(t)-\zb^*_{t-1}\|^2+1).
    \end{align}
Next, to get the right hand side above in terms of solely  $ \|\zb(t)-\zb^*_{t-1}\|^2$ we combine~\eqref{eq:CondExp1} with the following relations: 
 \begin{align*}
     &\|\bmu(t)-\ab_t^*\|\le \|\bmu(t)-\ab_{t-1}^*\| +\|\ab_{t}^* -\ab_{t-1}^*\|,\cr
     &\|\bmu(t)-\ab_{t-1}^*\| \le \|\bmu(t)-\ab_{t-1}^*\|^2 + 1,\cr
     &\|\blab(t)-\blab_t^*\|\le\|\blab(t)-\blab_{t-1}^*\| + \|\blab_{t-1}-\blab_t^*\|,\cr
     &\|\blab(t)-\blab_{t-1}^*\| \le  \|\blab(t)-\blab_{t-1}^*\|^2 + 1,\cr
     &\|\bmu(t)-\ab_t^*\|^2\le (1+0.5\gamma_t\nu)\|\bmu(t)-\ab_{t-1}^*\|^2\cr
     &\qquad\qquad\qquad+ \left(1+\frac{2}{\gamma_t\nu}\right)\|\ab_t^*-\ab_{t-1}^*\|^2,\cr
     &\|\blab(t)-\blab_t^*\|^2\le (1+0.5\gamma_t\varepsilon_t)\|\blab(t)-\blab_{t-1}^*\|^2 \cr
     &\qquad\qquad\qquad+ \left(1+\frac{2}{\gamma_t\varepsilon_t}\right)\|\blab^*_t-\blab_{t-1}^*\|^2,\cr
     &\|\ab^*_t-\ab^*_{t-1}\|^2 = O\left(\frac{(\varepsilon_t-\varepsilon_{t-1})^2}{\varepsilon_{t}}\right)  \mbox{(see Lemma~\ref{lem:t_vs_t-1}),} \cr
     &\|\blab^*_t-\blab^*_{t-1}\|^2 = O\left(\frac{(\varepsilon_t-\varepsilon_{t-1})^2}{\varepsilon^2_{t}}\right) \mbox{(see Lemma~\ref{lem:t_vs_t-1})}.
 \end{align*}
Thus, for sufficiently large $t$,% under the choice of the parameters,
 \begin{align}\label{eq:2}
 & \E\{\|\zb(t+1)-\zb_t^*\|^2|\EuScript F_t\}\le(1-\gamma_t\nu)\|\bmu(t)-\ab_{t-1}^*\|^2 \cr
 &\qquad\qquad\qquad+ (1-\gamma_t\varepsilon_t)\|\blab(t)-\blab_{t-1}^*\|^2\cr
 &+O\left(\frac{(\varepsilon_t-\varepsilon_{t-1})^2}{\gamma_t\varepsilon^3_{t}}+\gamma_t\sigma_t+\g_t^2\right)(\|\zb(t)-\zb_{t-1}^*\|^2+1)\cr
 &\le
  (1-\gamma_t\varepsilon_t)\|\zb(t)-\zb_{t-1}^*\|^2\\
  \nonumber
 &+O\left(\frac{(\varepsilon_t-\varepsilon_{t-1})^2}{\gamma_t\varepsilon^3_{t}}+\gamma_t\sigma_t+\g_t^2\right)(\|\zb(t)-\zb_{t-1}^*\|^2+1).
\end{align}
Given the settings of the time-dependent parameters namely $m=\min\{2-g-e,  g+s, 2g\}>1$, we have
    $\sum_{t=1}^{\infty}\frac{(\varepsilon_t-\varepsilon_{t-1})^2}{\gamma_t\varepsilon^3_{t}}+\gamma_t\sigma_t+\g_t^2<\infty$.
    With the inequality above in place, we are ready to apply Theorem~\ref{th:bound} to conclude  boundedness of $\|\zb(t)\|$ for any $t$ almost surely.
    
\textbf{Step 2. Convergence rate of $\E\|\blab(t+1)-\blab_{t}^*\|^2$.} 
%Next, we focus on the case $c=1$. Under the parameters' choice in this case, 
By taking the full expectation of both sides in~\eqref{eq:2}, we get, for sufficiently large $t$:
\begin{align}\label{eq:c=1}
&\E\|\zb(t+1)-\zb_t^*\|^2\le\left(1-\frac{G\nu}{t^{g}}\right)\E\|\bmu(t)-\ab_{t-1}^*\|^2 \cr
&\qquad+ \left(1-\frac{GE}{t^{g+e}}\right)\E\|\blab(t)-\blab_{t-1}^*\|^2+O\left(\frac{1}{t^{m}}\right)\cr
&\le\left(1-\frac{GE}{t^{g+e}}\right)\E\|\zb(t)-\zb_{t-1}^*\|^2+O\left(\frac{1}{t^{m}}\right).
\end{align}
%where $m = \min\{2-g-e,  g+s, 2g\}$. 
Thus, according to  Chung's lemma (see Lemma~\ref{lem:chung} in Appendix~\ref{app2}), we conclude that $\E\|\zb(t)-\zb^*_{t-1}\|^2 = O(1/t^{m-g-e})$. Thus, 
$$\E\|\blab(t+1)-\blab_{t}^*\|^2 = O(1/t^{m-g-e}).$$

\textbf{Step 3. Convergence rate of $\E\|\bmu(t)-\ab_{t-1}^*\|^2$.} Using this last relation together with the relation $\E\|\zb(t+1)-\zb_t^*\|^2\ge \E\|\bmu(t+1)-\ab_{t}^*\|^2+\left(1-\frac{GE}{t^{g+e}}\right)\E\|\blab(t+1)-\blab_{t}^*\|^2$, which holds if $\frac{GE}{t^{g+e}}\in[0,1]$, we get for sufficiently large $t$ that
\begin{align}\label{eq:c=1_1}
    &\E\|\bmu(t+1)-\ab_{t}^*\|^2\le\left(1-\frac{G\nu}{t^{g}}\right)\E\|\bmu(t)-\ab_{t-1}^*\|^2 \cr
    &+ \left(1-\frac{GE}{t^{g+e}}\right)\left( O\left(\frac{1}{t^{m-g-e}}\right)-O\left(\frac{1}{(t+1)^{m-g-e}}\right)\right)\cr
&\qquad\qquad\qquad+O\left(\frac{1}{t^{m}}\right)\cr
&= \left(1-\frac{G\nu}{t^{g}}\right)\E\|\bmu(t)-\ab_{t-1}^*\|^2 +O\left(\frac{1}{t^{m}}\right),
\end{align}
since $O\left(\frac{1}{t^{m-g-e}}\right)-O\left(\frac{1}{(t+1)^{m-g-e}}\right) = O\left(\frac{1}{t^{1+m-g-e}}\right)$ and $e+g<1$. Thus, applying  Chung's Lemma~\ref{lem:chung} to~\eqref{eq:c=1_1}, we obtain \begin{align}\label{eq:rate1}
    \E\|\bmu(t)-\ab_{t-1}^*\|^2 = O\left(\frac{1}{t^{m-g}}\right).
\end{align}

\textbf{Step 4. Convergence rate of $\E\|\bmu(t)-\ab^*\|^2$.}
Finally,~\eqref{eq:rate1} together with Lemma~\ref{lem:epsApprox} implies the result of the theorem, namely, 
$
    \E\|\bmu(t)-\ab^*\|^2 =O\left(\frac{1}{t^{\min\{2e,m-g\}}}\right).
$

\end{proof}
Optimizing the parameters $\gamma_t$, $\varepsilon_t$, and $\sigma_t$ in Theorem~\ref{th:main}, we get the following result for the convergence rate.
\begin{corrolary}\label{cor:1}
Let Assumptions~\ref{assum:convex}-\ref{assum:infty} hold. In Algorithm~\eqref{eq:alg} choose the parameters  as follows:
$\gamma_t = \frac{G}{t^{4/7}}$, $\varepsilon_t=\frac{E}{t^{2/7}}$, $\sigma_t=\frac{S}{t^s}$, $s\ge 4/7$.
%where $\delta$ is chosen to be arbitrary small. 
Moreover, let the joint action $\ab^*$ be the unique variational generalized Nash equilibrium of the game $\Gamma$. 
Then for the iterates in Algorithm~\eqref{eq:alg} the following holds: 
\begin{align*}
    \E\|\bmu(t)-\ab^*\|^2 =O\left(\frac{1}{t^{4/7}}\right).
\end{align*}
\end{corrolary}
\begin{remark}
    Observe that the above result is the first convergence rate estimation in zeroth-order learning of GNE. Unfortunately (and naturally) it is worse than the best rate $O(1/t)$ attainable in the zeroth-order two-point feedback strongly monotone  games with uncoupled constraints ~\cite{TatKamECC24}. 
    We emphasize here that establishing convergence rates of algorithms converging to GNE in games is a challenging problem. The work~\cite{jordan2023first} proves the rate $O(1/t^2)$ for the case of strongly monotone games with affine constraints, under exact first-order information given a specifically defined surrogate of a GNE solution; while the best rate in strongly monotone games with uncoupled constraints under the first-order information settings is $O\left(\exp\{-\frac{t}{\gamma^2}\}\right)$ with $\gamma = L/\nu$, see~\cite{NesterovScrimali}. 
\end{remark}

%\maryam{A comment on the rates can be compared to XXX in the full information feedback (\cite{meng2023linear}?) and also XXXX in zeroth-order constrained optimization REF? }

\section{Numerical example}
\label{sec:example}

%\textbf{Example 1.}
Let the uncoupled and coupled action sets in a two-player game be defined as follows: $ A^1 = A^2 = \R$,  and $C = \{\ab=[a^1,a^2]:\, a^1+a^2\ge 1\}$ respectively, whereas the cost functions are $J^1(a^1,a^2) = \frac{3}{2}(a^1)^2 + a^1a^2$ and $J^2(a^1,a^2) = \frac{1}{2}(a^2)^2 - a^1a^2$. The game satisfies Assumptions~\ref{assum:convex}-\ref{assum:infty}. Moreover, the variational GNE is the joint action $\ab^* = [0,1]$ with the corresponding dual variable $\lambda^* = 1$ (a solution to the game $\Gamma^{\blab}$). 
\begin{figure}[!htb]
	\centering
	\includegraphics[width=0.5\textwidth]{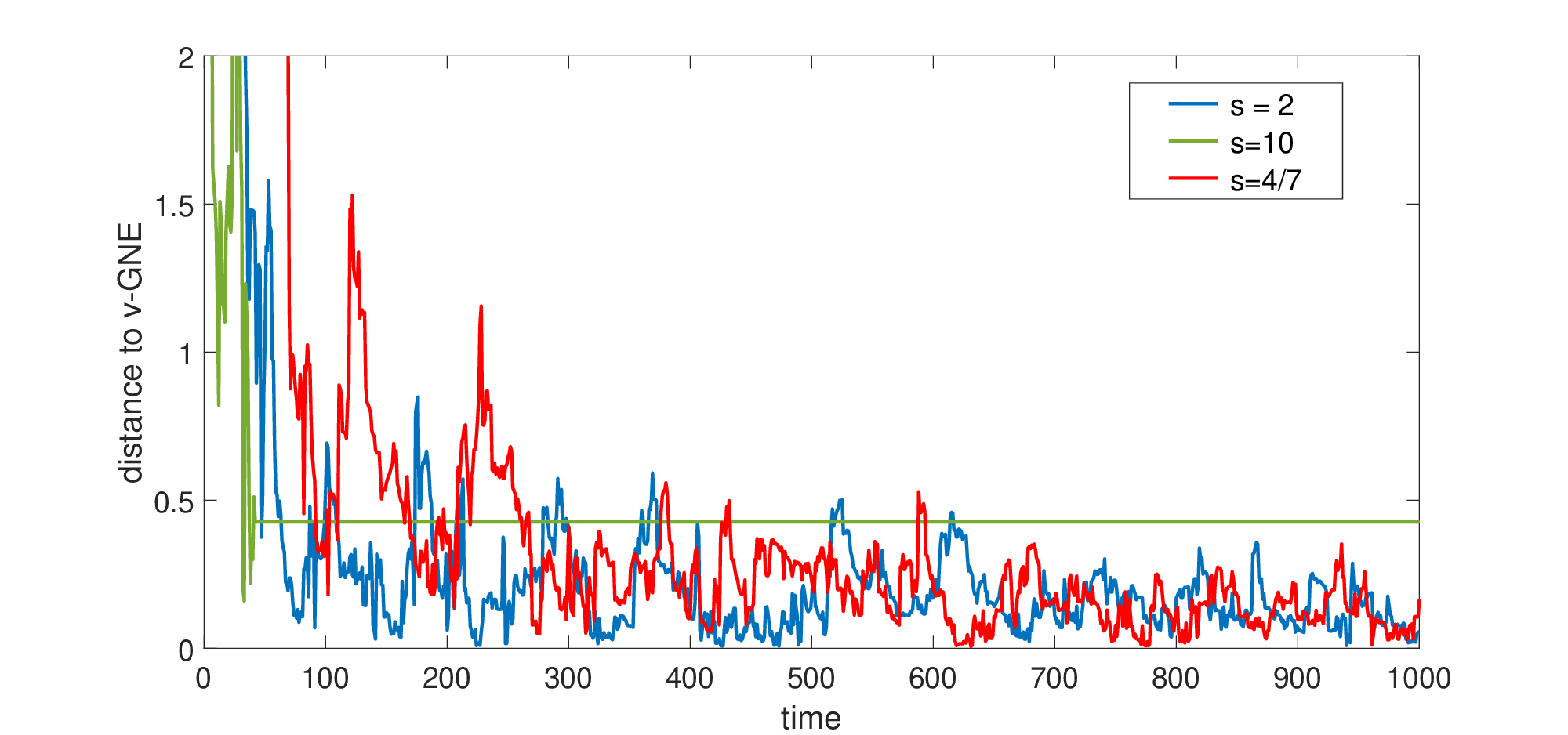}
	\caption{\label{eps:compare}Convergence of Algorithm~\eqref{eq:alg} for different settings of $\sigma_t$.}
\end{figure}
We implemented the procedure of Algorithm~\eqref{eq:alg} for the following settings of the parameters: $\gamma_t = \frac{1}{t^{4/7}}$, $\varepsilon_t = \frac{1}{t^{2/7}}$, and $\sigma_t = \frac{1}{t^{s}}$, where $s = 4/7$ (red line), $s = 2$ (blue line), and $s = 10$ (green line). All of these parameters satisfy the conditions of Theorem~\ref{th:main}. As we can see, Algorithm~\eqref{eq:alg} demonstrates convergence in the cases $s=4/7$ and $s = 2$, while for the case $s=10$ the procedure does not make any significant progress after approximately 50 iterations. This is due to a very small variance $\sigma_t$, resulting in the sampled action to be almost equal to the mean value ($a^i = \mu^i$). Consequently, it leads to extremely small value of the gradient estimate $m^i$. Hence, the optimal tuning of  $\sigma_t$ is an important direction of future work. 

\section{Conclusion}
We derived the convergence rate of a payoff-based approach to learn the variational equilibrium in  a class of GNE problems. The considered class consists of games with strongly monotone pseudo-gradient and jointly linear coupling constraints. Given that the current work is the first one characterizing the convergence rate of payoff-based learning in GNE problems, future work should explore potential improvements in the convergence rate. Furthermore, future research should explore how to optimally tune the variance of the sampling approach in the gradient estimation procedure of the algorithm.

\bibliographystyle{plain}
\bibliography{srtrMonGames_ref}

\appendix
\subsection{Proof of Lemma~\ref{lem:exist_uncoupled}}\label{app:exist_uncoupled}
\begin{proof} According to the definition of variational GNE (see Definition~\ref{def:VGNE}), $\ab^*$ is variational GNE if and only if it solves~\eqref{eq:GVI}.  
Thus,  assertion 1) holds due to Proposition 1.3.4. in~\cite{FaccPang1}. 

The existence of a unique variational GNE in $\Gamma$ is guaranteed under Assumption~\ref{assum:strmon}, as there exists a unique solution to~\eqref{eq:GVI} if the mapping $\Mb$ is strongly monotone. Thus,  assertion 3) is implied. Moreover, given Assumptions~\ref{assum:convex} and~\ref{assum:Slaters} and taking into account assertion 1),  assertion 2) follows. 
Finally,  assertion 4) holds under the Slater's constraint qualification condition in Assumption~\ref{assum:Slaters}, as demonstrated in Lemma~5.1 in \cite{ZhuFrazzoli}.
\end{proof}

\subsection{Proofs of Lemmas~\ref{lem:epsApprox} and~\ref{lem:t_vs_t-1}}\label{app1}
\begin{proof}(Proof of Lemma~\ref{lem:epsApprox}.)
    According to the definition of $\Wb_t$ and  given $\zb^* = [\ab^*, \blab^*]$, $\zb^*_t =  [\ab^*_t, \blab^*_t]$, we have
    \begin{align*}
        &\langle \Wb(\zb^*),\zb^*_t-\zb^*\rangle\ge 0,\cr
        &\langle \Wb_t(\zb_t^*),\zb^*-\zb_t^*\rangle=\langle \Wb(\zb_t^*)+\varepsilon_tI_{0,n}\zb_t^*,\zb^*-\zb_t^*\rangle\ge 0.
    \end{align*}
    By summing up the inequalities above and using~\eqref{eq:aug-psedo}, which holds due to Assumption~\ref{assum:strmon}, we obtain
    \begin{align}\label{eq:sum}
        0\le -\nu\|\ab^*-\ab^*_t\|^2 +\varepsilon_t\langle\blab^*_t,\blab^*-\blab^*_t\rangle,
    \end{align}
    which implies
    $
        \nu\|\ab^*-\ab^*_t\|^2 +\varepsilon_t\|\blab^*-\blab^*_t\|^2\le\varepsilon_t\langle\blab^*\blab^*-\blab^*_t\rangle
    $.
    Thus, 
    $
        \varepsilon_t\|\blab^*-\blab^*_t\|^2\le\varepsilon_t\|\blab^*\|\|\blab^*-\blab^*_t\|$,
       implying 
   $\|\blab^*-\blab^*_t\|\le\|\blab^*\|$. Hence, from Lemma \ref{app:exist_uncoupled} assertion 4), we have uniform boundedness of $\|\blab^*-\blab^*_t\|$. Moreover, as $\blab^* = \Proj_{\R^n_{+}}[\blab^* + \alpha \gb(\ab^*)]$, $\blab_t^* = \Proj_{\R^n_{+}}[\blab_t^* + \alpha (\gb(\ab_t^*) + \varepsilon_t\blab_t^*)]$ for any $\alpha\ge 0$, and the projection is a continuous mapping, we conclude that  $\lim_{\varepsilon_t\to 0}\|\blab^*-\blab^*_t\| = 0$, where $[\ab^*,\blab^*]$ is  the minimal norm solution to~\eqref{eq:VI1}. On the other hand,~\eqref{eq:sum} together with the previous relation implies
   \begin{align}\label{eq:app1_1}
        \nu\|\ab^*-\ab^*_t\|^2 \le\varepsilon_t\|\blab^*\|\|\blab^*-\blab^*_t\|\le\varepsilon_t\|\blab^*\|^2.
    \end{align}
    On the other hand, according to~\eqref{def:mapping_ext} and~\eqref{eq:primalM} and due to the fact that $[\ab^*,\blab^*]$ solves~\eqref{eq:VI1}, we conclude that
$\Wb^{pr}(\ab^*, \blab^*) = \Mb(\ab^*) + K^T\blab^* = \boldsymbol{0}$.
Moreover,  
$\Wb^{pr}(\ab_t^*) = \Mb(\ab_t^*) + K^T \blab_t^* = \boldsymbol{0}$, as $[\ab_t^*,\blab_t^*]$ solves~\eqref{eq:VI2}.
Hence, 
$\Mb(\ab^*) + K^T\blab^* = \Mb(\ab_t^*) + K^T\blab_t^*$, which implies that 
\begin{align*}
    \|\blab^*-\blab_t^*\|
    \|K\| = \|\Mb(\ab^*) - \Mb(\ab_t^*)\| \le L\|\ab^*-\ab_t^*\|,
\end{align*}
where the last inequality is due to existence of the Lipschitz constant $L$ for the mapping $\Mb$ (see Remark~\ref{rem:Lip}). Next, using~\eqref{eq:app1_1}, we conclude that 
\[\|\ab^*-\ab^*_t\|^2 \le\varepsilon_t\frac{\|\blab^*\|\|\blab^*-\blab^*_t\|}{\nu}\le\varepsilon_t\frac{\|\blab^*\|L\|\ab^*-\ab_t^*\|}{\|K\|\nu}.\]
Thus, $\|\ab^*-\ab^*_t\|\le \varepsilon_t\frac{\|\blab^*\|L}{\|K\|\nu}$.
\end{proof}

\begin{proof}(Proof of Lemma~\ref{lem:t_vs_t-1}.) According to 
the definition of $\Wb_t$ and  given $\zb^*_t = [\ab^*, \blab^*]$, we have
    \begin{align*}
        &\langle \Wb(\zb_{t-1}^*)+\varepsilon_{t-1}I_{0,n}\zb_{t-1}^*,\zb^*_t-\zb_{t-1}^*\rangle\ge 0,\cr
        &\langle \Wb(\zb_t^*)+\varepsilon_tI_{0,n}\zb_t^*,\zb_{t-1}^*-\zb_t^*\rangle\ge 0.
    \end{align*}
    By summing up the inequalities above and using~\eqref{eq:aug-psedo}, we obtain
    \begin{align}\label{eq:sum2}
        &0\le-\nu\|\ab^*_{t-1}-\ab^*_t\|^2 -\varepsilon_t\|\blab^*_{t-1}-\blab^*_t\|^2 \cr
        &\qquad+ (\varepsilon_{t-1}-\varepsilon_t)\langle \blab_{t-1}^*,\blab^*_t-\blab_{t-1}^*\rangle,
    \end{align}
    which implies, on the one hand,
    $\varepsilon_t\|\blab^*_{t-1}-\blab^*_t\|^2 \le |\varepsilon_{t-1}-\varepsilon_t| \|\blab_{t-1}^*\|\|\blab^*_t-\blab_{t-1}^*\|$.
       Thus, due to the boundedness of $\|\blab_{t-1}^*\|$ (see the inequality $\|\blab^*-\blab^*_t\|\le\|\blab^*\|$ in the proof of Lemma~\ref{lem:epsApprox}), we conclude that
    $\|\blab^*_t-\blab^*_{t-1}\|^2 = O\left(\frac{(\varepsilon_t-\varepsilon_{t-1})^2}{\varepsilon^2_{t}}\right)$.
    On the other hand, the relation above together with~\eqref{eq:sum2}, implies
    $\|\ab^*_t-\ab^*_{t-1}\|^2 = O\left(\frac{(\varepsilon_t-\varepsilon_{t-1})^2}{\varepsilon_{t}}\right)$.
\end{proof}

\subsection{Proof of Lemma~\ref{lem:Rsq}}\label{app:Rsq}
\begin{proof}
According to Lemma~\ref{lem:sample_grad}, 
     \begin{align}\label{eq:R_2}
        \E\{\|\Rb^i(t)\|^2|\EuScript F_t\}\le\E\{\|{\mb}^i(t)\|^2|\EuScript F_t\}.
    \end{align}
    Moreover, the definition of ${\mb}^i(t)$ in~\eqref{eq:est_Gd2} implies that 
    \begin{align*}
        &\E\{\|{\mb}^i(t)\|^2|\EuScript F_t\} = \E\{(U^{i}(\ab(t),\blab(t)) - U^{i}(\bmu(t),\blab(t)))^2\cr
        &\qquad\qquad\times\frac{\|\ab^i(t) -\bmu^i(t)\|^2}{\sigma^4_t}|\EuScript F_t\}\cr
        &\le\frac{\E\{(U^{i}(\ab(t),\blab(t)) - U^{i}(\bmu(t),\blab(t)))^2|\EuScript F_t\}}{\sigma_t^2},
    \end{align*}
    where we used the H\"older's inequality. Next, applying the Mean Value Theorem, we conclude that 
    \begin{align}
    &\E\{(U^{i}(\ab(t),\blab(t)) - U^{i}(\hbmu(t),\blab(t)))^2|\EuScript F_t\}\cr
    & = \E\{(\langle\nabla_{\ab}U^{i}(\boldsymbol{\theta}(t),\blab(t)), \ab(t) - \hbmu(t)\rangle)^2|\EuScript F_t\}\cr
  %  &\le \E\{\|\langle\nabla_{\ab}U^{i}(\boldsymbol{\theta}(t),\blab(t))\|^2 \|\ab(t) - \hbmu(t)\|^2|\EuScript F_t\}\cr
    &\le \E\{\|\langle\nabla_{\ab}U^{i}(\boldsymbol{\theta}(t),\blab(t))\|^2|\EuScript F_t\}\sigma_t^2,
    \end{align}
    where $\boldsymbol{\theta}(t) = \ab(t) + \theta(\hbmu(t)-\ab(t))$ for some $\theta\in[0,1]$, the inequality is again due to the Cauchy-Schwarz and H\"older inequalities, respectively. Finally, taking into account Assumption~\ref{assum:infty} and the definition of the function $U^i$ (see~\eqref{eq:J^{pr}r}), we conclude that 
    \begin{align}\label{eq:R_2f}
        \E\{\|\nabla_{\ab}U^{i}(\boldsymbol{\theta}(t),\blab(t))\|^2|\EuScript F_t\} = f^i(\bmu(t),\blab(t)),
    \end{align}
    where $f^i(\bmu(t),\blab(t))$ is quadratic in both $\bmu(t)$ and $\blab(t)$.
    Thus, we conclude the result from~\eqref{eq:R_2}-\eqref{eq:R_2f}.
\end{proof}
\subsection{Theorem 2.5.2 in \cite{NH}}\label{app_bound}
\begin{theorem}\label{th:bound}
Let $V(t,X)$ be a real-valued function defined for some  stochastic process $\{X(t)\}_t$ taking values in $\R^d$. Let us assume existence of the following generating operator 
\begin{align*}
LV(t,x)=E[V(t+1, X(t+1))\mid X(t)=x]-V(t,x),
\end{align*}
for which the following decay holds: 
\begin{align*}
LV(t,x)\le -\alpha(t+1)\psi(x) + \phi(t)(1+V(t,x)),
\end{align*}
where $\psi\ge 0$ on $ \R^{d}$, $\phi(t)>0$, $\forall t$, $\sum_{t=0}^{\infty}\phi(t)<\infty$, $\alpha(t)>0$, $\sum_{t=0}^{\infty} \alpha(t)= \infty$. Then the process $X(t)$ is almost surely bounded over $t$.
\end{theorem}

\subsection{Chung's Lemma}\label{app2}
Here we formulate the result from Lemma 4 in \cite{Chung}.
\begin{lem}\label{lem:chung}
Suppose that $\{b_n\}$, $n\ge1$, is a sequence of real numbers such that for $n\ge n_0$,
\[b_{n+1}\le\left(1-\frac{c_n}{n^s}\right)b_n+\frac{c'}{n^t},\]
where $0<s<1$, $c_n\ge c>0$, $c'>0$. Then
\[\limsup_{n\to\infty}n^{t-s}b_n\le \frac{c'}{c}.\]
\end{lem}
\end{document}